\documentclass[11pt]{article}
\usepackage{amsmath}
\usepackage{amsthm}

\usepackage{amssymb}
\usepackage{mathrsfs}
\usepackage{latexsym}
\theoremstyle{plain}
\newtheorem{thm}{Theorem}[section]
\newtheorem{lem}[thm]{Lemma}
\newtheorem{cor}[thm]{Corollary}
\newtheorem{prop}[thm]{Proposition}
\theoremstyle{definition}
\newtheorem{rem}[thm]{Remark}
\newtheorem{defn}[thm]{Definition}

\numberwithin{equation}{section}

\newcommand{\R}{\mathbb{R}}
\newcommand{\C}{\mathbb{C}}

\newcommand{\pa}{\partial}

\newcounter{kotaeflg}
\newcommand{\kotae}[1]{
\ifodd \arabic{kotaeflg}
#1
\fi
}

\begin{document}




\title{Unconditional well-posedness for the nonlinear Schr\"odinger equation in Bessel potential spaces}
\author{Ryosuke Hyakuna \\
  {\normalsize Polytechnic University of Japan. }\\
{\normalsize Kodaira-city, 169-8555, Tokyo, Japan}\\
{\normalsize email: 107r107r@gmail.com}}
%
%
\maketitle
\abstract{The Cauchy problem for the nonlinear Schr\"odinger equation is called unconditionally well posed in a data space $E$ if it is well posed in the usual sense and the solution is unique in the space $C([0,T]; E)$.  In this paper, this notion of unconditional well-posedness is redefined so that it covers $L^p$-based Sobolev spaces as data space $E$ and it is equivalent to the usual one when $E$ is
an $L^2$-based Sobolev space $H^s$. Next, based on this definition, it is shown that the Cauchy problem for the 1D cubic NLS is unconditionally well posed in Bessel potential spaces $H^s_p$ for
$4/3<p\le 2$ under certain regularity assumptions on $s$. }


\section{Introduction}
We consider the Cauchy problem for the one dimensional cubic nonlinear Schr\"odinger equation
\begin{equation}
iu_t+u_{xx}+|u|^2u =0, \quad u(0)=\phi\in E,\label{NLS}
\end{equation}
where $E$ is a Banach space of complex valued functions on $\R$.  Recall that Cauchy problem (\ref{NLS}) is
called locally well posed in the data space $E$ if, for any $\phi \in E$ there are a $T=T(\phi )>0$ and a unique solution
$u$ of (\ref{NLS}) in the space $C([0,T]; E)\cap Y_T\triangleq Z_T$, where $Y_T$ is a space of functions
on $[0,T]\times \R$.  For example, (\ref{NLS}) is locally (and globally) well posed in $H^s$ for $s\ge 0$. The space $Y_T$ is called an auxiliary space and in some cases the well-posedness can be shown without 
using this space.  In fact, if $s>1/2$ the solution can be established directly in a closed subset of $L^{\infty}([0,T] ;H^s)$ via the 
standard fixed point argument and the local well-posedness holds with $Z_T=C([0,T] ; H^s)$.  This is possible because the space $H^s$ forms an algebra if $s>1/2$.  
On the other hand, when $s<1/2$, the solution of (\ref{NLS}) is usually constructed in the space $L^{\infty}([0,T]; H^s) \cap L^q([0,T] ; H_r^s)$ for suitable choices of $(q,r)$ with the aid of 
the Strichartz estimate (see e.g. \cite{Caz}).  In this case, $Y_T=L^q([0,T] ; H_r^s)$, and the uniquness assertion is shown in this space, not in $C([0,T] ;H^s)$. Therefore, 
it is natural to ask if the well-posedness holds with $Z_T=C([0,T] ;E)$ for $s<1/2$.  
This problem of whether or not the auxiliary condition is removable was first proposed by Kato \cite{KatoHs}.  He said (\ref{NLS}) is \textit{unconditionally well posed} in $E$ if 
it is well posed with $Z_T=C([0,T] ; E)$, or equivalently, the uniqueness of solutions holds in the space $C([0,T] ; E)$ in addition to its well-posedness.  In \cite{KatoHs} he showed that
the unconditional well-posedness in $H^s$ can be pushed from the trivial case $s>1/2$ down to $s=1/6$.  It is known that this is the minimal Sobolev regularity
for the unconditional well-posedness of the 1D cubic NLS (\ref{NLS}).  Thus, the fact that $u\in L^q([0,T] ; H_r^s)$ is simply an additional regularity property of the solution which is removable for $s\ge 1/6$, while
it is required to ensure the well-posedness in $H^s$ for $s<1/6$. Since the pioneering work \cite{KatoHs}, the problem of the unconditional well-posedness for nonlinear Schr\"odinger equations
and other nonlinear dispersive equations has been extensively studied.  We refer to \cite{Kishimoto} and references therein for earliear results in this direction.

The aim of the paper is to extend the unconditional well-posedness results for the $L^2$-based Sobolev spaces $H^s$ to the Bessel potential spaces $H^s_p$.  One might think discussing such a problem is meaningless, since
if $p\neq 2$ the Schr\"odinger equations are not well posed in $H^s_p$ even in the linear case (see \cite{Brenner,Hormander}).  In particular, we cannot expect the persistence property $u(t) \in H^s_p$ for data $u(0)\in H^s_p$ unless $p=2$ let alone the uniqueness issue in the space $C([0,T] ; H^s_p)$.  Nevertheless, we stress that the problem of unconditional well-posedness can be generalized to the $L^p$-setting in a natural manner.  Our idea here is motivated by Zhou.  In \cite{Zhou} he introduced the ``twisted" variable $v(t):=e^{-it\pa_x^2}u(t)$ and rewrite the integral equation corresponding to (\ref{NLS}) as
\begin{equation}
v(t)=\phi+i\int^t_0 e^{-i\tau \pa_x^2} \left[ (e^{i\tau \pa_x^2} v(\tau))(e^{i\tau \pa_x^2} v(\tau)) (\overline{ e^{i\tau \pa_x^2} v(\tau) }) \right] d\tau. \label{INT}
\end{equation}
Then he showed that a local solution $v$ of (\ref{INT}) exists in the space $C([0,T]\, ;\, L^p)$ for $\phi
\in L^p,\,1<p<2$.  This result implies the existence of a solution of the original 
Cauchy problem (\ref{NLS}) such that $e^{-it\pa_x^2} u(t) \in C([0,T]\,;L^p)$.  Thus, a suitable space 
to discuss the problem of the unconditional uniqueness in the $L^p$-space is 
\begin{equation}
\{ u\,|\, e^{-it\pa_x^2} u(t) \in C([0,T]\, ; L^p)\}.
\end{equation}
In fact, it is easy to see that if $p=2$, this space coincides with $C([0,T];L^2)$.  In
this paper, we formulate the notion of the unconditional well-posedness and uniqueness
in the $L^p$, and more generally, the Bessel potential spaces $H^s_p$ based on this idea. Then, we show that
(\ref{NLS}) is unconditionally locally well posed in $H^s_p$ under some assumptions on the Sobolev regularity.
Finally, for the unconditional well-posedenss in other non-$L^2$-based spaces, we refer to
\cite{OW}, where they discuss the unconditional uniqueness of the periodic NLS in the Fourier-Lebesgue spaces $\mathcal{F}L^p$ as data space.  Note also that
for the Fourier-Lebesgue spaces, one does not need to consider the twisted variable as in the case of $H^s_p$, since the solution is expected to
have the usual persistence property $u\in C([0,T] ; \mathcal{F}L^p)$.  See \cite{Grunrock}.

\medskip

\textbf{Notation}. \, For $1\le p \le \infty$ and $I \subset \R$, $\|f\|_{L^p(I)}$ denotes the usual $L^p$-norm.  \,$p'$ is the conjugate
exponent of $p$: $1/p+1/p'=1$.  The Fourier transform of $\phi$ is denoted by $\hat{\phi}$.  The Bessel potential spaces are defined by
\begin{equation*}
H^s_p (\R)
\triangleq \{\phi \in\mathscr{S}'(\R)\,|\, \langle D \rangle^s \phi \in L^p (\R)\},
\end{equation*}
where $\widehat{\langle D \rangle^s \phi} (\xi)=(1+|\xi|^2)^{s/2} \hat{\phi}(\xi)$.  In particular, we 
write $H^s=H^s_2$ as usual.  For simplicity we often write $L^p, H^s_p$ to denote
$L^p(\R), H^s_p(\R)$ respectively.  For a function $u:I\times \R \to \C$, we set
\begin{equation*}
\|u\|_{L^q(I; H^s_p )} \triangleq \left(\int_I \left\| u (t, \cdot) \right\|^q_{H^s_p} dt \right)^{\frac{1}{q}}.
\end{equation*}

\medskip

We define function spaces $\mathfrak{C}( I ; E)$ and $\|u\|_{\mathfrak{L}^{\infty} ( I ;E)}$ to introduce a concept of
well-posedness of (\ref{NLS}) in $L^p$.

\begin{defn} 
Let $E\subset \mathcal{S}'(\R)$ and $I\subset \R$. The space $\mathfrak{C}(I; E)$ is defined by
\begin{equation*}
\mathfrak{C}( I ; E) \triangleq \{ u=e^{it\pa_x^2} v(t)\,|  \, v \in  C( I ;E) \}.
\end{equation*}
Similarly, the space $\mathfrak{L}^{\infty} (I ; E)$ is defined by
\begin{equation*}
\mathfrak{L}^{\infty}( I ; E) \triangleq \{ u=e^{it\pa_x^2} v(t)\,|  \, v\in  L^{\infty}( I ;E) \}.
\end{equation*}
equipped with the norm
\begin{equation*}
\|u\|_{\mathfrak{L}^{\infty} ( I ;E)} \triangleq \| e^{-it\pa_x^2} u \|_{L^{\infty}(I ; E)}.
\end{equation*}

\end{defn}

We first give the precise definition of ``twisted" local well-posedness for (\ref{NLS}) which was essentially introduced by Zhou.

For a normed space $E$, $\mathcal{B}_E(r)$ denotes the closed ball with the radius $r>0$, namely
$\mathcal{B}_E(r)\triangleq \{ u\in E \,|\, \| u\|_E \le r\}$.

\begin{defn}\label{TLWP}
Let $E$ be a normed space of functions on $\R$.  Cauchy problem (\ref{NLS}) is locally well posed in $E$ if, for any $M>0$ there exist a $T_M>0$ and a space $Y_{T_M}$ of functions on $[0,T_M]\times \R$ such that: for any $\phi \in \mathcal{B}_E(M)$ there is a unique solution in $u\in Z_{T_M}\triangleq \mathfrak{C}
([0,T_M] ; E) \cap Y_{T_M}$.  Moreover, the map $\phi \mapsto u$ is continuous from
$\mathcal{B}_E(M)$ to $Z_{T_M}$.  
\end{defn}

Now we give the definition of the unconditional well-poseness for (\ref{NLS}).

\begin{defn}\label{ULWP}
Let $E$ be the same as Definition \ref{TLWP}.  Cauchy problem (\ref{NLS}) is unconditionally locally well posed in $E$ if, it is well posed in the sense of Definition \ref{TLWP} with $Z_{T_M}=\mathfrak{C}([0,T_M] ;E)$, namely the uniqueness hold in the space $\mathfrak{C}([0,T_M] ;E)$.
\end{defn}

\begin{rem}
Note that in the case of $E=H^s$, (\ref{NLS}) is locally well posed in the sense of Definition \ref{TLWP} if 
and only if it is locally well posed in the usual sense.  Similarly, it is unconditionally locally well posed
in the sense of Definition \ref{ULWP} if and only if it is unconditionally locally well posed in the usual sense.

\end{rem}

We begin with the ``conditional" well-posedness result.

\begin{prop} \label{LWP}
Let $4/3<p\le 2$ and $0<s<3/2-2/p$.  Then (\ref{NLS}) is locally well posed
in $H_p^s(\R)$ in the sense of Definition \ref{TLWP}.

\end{prop}

The conditional well-posedness results in $H^s_p$ can be proved arguing
similarly as in the proof of \cite[Theorem 1.1]{107jfa}.  So we only give a sketch of
proof of this proposition at the end of the paper.

The main results of this paper is the unconditional well-posedness for the Cauchy problem
under additional regularity assumptions on data:

\begin{thm}\label{ULWPTH}
Let $4/3<p\le 2$ and $s<3/2-2/p$.  Then (\ref{NLS}) is unconditionally locally
well posed in $H_p^s(\R)$ in the sense of {\rm Definition \ref{ULWP}} if:
\begin{enumerate}
\item
$4/3<p \le 3/2$ and $s>0$.
\item
$3/2<p  \le 2$ and $s> 2/3-1/p$.
\end{enumerate}

\end{thm}

\begin{rem}
Let $s_c(p)\triangleq \max(0,2/3-1/p)$ for $1\le p \le 2$. Then Theorem \ref{ULWPTH} asserts that
(\ref{NLS}) is unconditionally locally well posed in $H_p^s$ for $s>s_c(p)$ and $4/3<p\le 2$.  The exponent $s_c(p)$ is considered 
as a natural threshold for the unconditional well-posedness in the following sense.  In the case of the cubic NLS, one
need $|u|^2u \in L_{loc}^1$ so that the nonlinear part makes sense in the distributional framework.  Thus we 
require $u \in L_{loc}^3$.  By the well-known time decay property of the evolution group $e^{it\pa_x^2}$, we have
$u(t) \in H^s_{p'},\,t\neq 0$ if $e^{-it\pa_x^2} u(t) \in H^s_p$ for $1\le p\le 2$.  When $3/2<p\le 2$, we see that $u(t) \in L^3$ if
$s\ge 2/3-1/p=s_c(p)$ by Sobolev's embedding.  When $1<p\le 3/2$, we assume $e^{-it\pa^2}u(t) \in L^p$.  Then
$u(t) \in L^{p'}$ and thus we have $u(t) \in L_{loc}^3$ since $3\le p ' \le \infty$. 
\end{rem}

\section{Proof of the main results}
\subsection{Key lemma and proposition}
Our unconditional well-posedness results can be proved by a basic embedding theorem for $\mathfrak{L}^{\infty}(I ; H^s_p)$ and
Strichartz type estimates.

\begin{lem} \label{Sobolev}
Let $I\subset \R$ be a finite interval.  Let $2\le q ,r \le \infty,\, 1\le p \le 2, s\ge 0$ satisfy
\begin{equation*}
 s\ge 1-\frac{1}{p}-\frac{1}{r},\quad q \left(\frac{1}{p}-\frac{1}{2} \right) <1\,( \text{ with the convention that} \,\,\infty \cdot 0=0).
\end{equation*}
Then
\begin{equation*}
\| f\|_{L^q(I; L^r(\R))} \le C_I \| f\|_{\mathfrak{L}^{\infty} ( I ;H^s_p (\R))}.
\end{equation*}
\end{lem}

\begin{proof} Let $2\le r \le \infty$.  We first fix $t\in I$.  Then, by Sobolev's embedding and the well-known decay property of
the free group $e^{it\pa_x^2}$, we have
\begin{equation*}
\|e^{it\pa_x^2} f(t) \|_{L^r} \le C \|e^{it\pa_x^2} f(t)\|_{\dot{H}^{s(r)}_{p'}} \le C (4\pi |t|)^{-(1/p-1/2)} \|f(t)\|_{\dot{H}^{s(r)}_p},
\end{equation*}
where $s(r)=1-1/p-1/r$.  Replacing $f$ with $e^{-it\pa_x^2}f$, we obtain
\begin{equation}
\| f(t) \|_{L^r} \le (4\pi |t|)^{-(\frac{1}{p}-\frac{1}{2})} \|e^{-it\pa_x^2}f(t)\|_{H^s_p},
\end{equation}
for any $s\ge s(r)$ and $t\in I$.  The desired embedding is obtained after taking $\|\cdot\|_{L^{q}(I)}$-norm of both sides and applying H\"older's inequality in the 
right hand side.

\end{proof}

The next key estimates are the inhomogeneous Strichartz inequalities for not necessarily admissible pairs.

\begin{prop} {\rm (See, e.g. \cite[Theorem 2.1]{Kato} )} \label{inhomostr}
Let $2\le q,r \le \infty\,\,,1<\gamma,\rho \le 2$ be such that
\begin{equation*}
2+\frac{2}{q}+\frac{1}{r}=\frac{2}{\gamma}+\frac{1}{\rho}
\end{equation*}
and 
\begin{equation*}
\frac{1}{q}+\frac{1}{r}<\frac{1}{2},\quad \frac{3}{2}-\frac{1}{\rho} <\frac{1}{\gamma}<1.
\end{equation*}
Then
\begin{equation}
\left\| \int^t_0 e^{i(t-\tau)\pa_x^2} F(\tau) d\tau \right\|_{L^q( I; L^r(\R))} \label{inhomostrest}
\le \| F\|_{L^{\gamma}( I; L^{\rho}(\R))}.
\end{equation}

\end{prop}

\subsection{Proof of Theorem \ref{ULWP}} We prove the unconditional well-posedness results.  We assume that the conditional local well-posedness result (Proposition \ref{LWP}) is given, which will be proved in the last section.  Then it is enough to
show the uniqueness of the solution in the space $\mathfrak{C}([0,T]; H^s_p)$ to conclude the unconditional well-posedness. We first consider case (i) and we let $4/3<p<3/2$ and $s>0$.  We may assume $s$ is sufficiently small, since
the uniquness of the solution in a function space also implies the uniquenss in any smaller spaces.  Let $\delta>0$ and we set
\begin{equation*}
\frac{1}{q}=\frac{1}{p}-\frac{1}{2} +\frac{\delta}{2},\quad \frac{1}{r}=1-\frac{1}{p}-\delta,\quad
\frac{1}{\rho}=3-\frac{3}{p}-3\delta(=\frac{3}{r}),\quad \frac{1}{\gamma}=\frac{2}{p}-\frac{1}{2}+\frac{3}{2}\delta.
\end{equation*}
Then the pair $(q,r)$ satisfies the assumption of Lemma \ref{Sobolev} with $s=\delta$ and the quadruple $(q,r,\gamma,\rho)$ satisfies the assumption of Proposition \ref{inhomostr} if
$\delta$ is sufficiently small.  Now we consider two solutions $u,v\in \mathfrak{C}([0,T];H^{\delta}_p(\R))$ of (\ref{NLS}) with $u(0)=v(0)$.  We want to show that $u(t)=v(t)\,\,\forall t\in [0,T]$.  By Lemma \ref{Sobolev} we see that
$u,v\in L^q([0,T] ; L^r)$.  Let $0<T_0 \le T$.  We estimate the difference $\|u-v\|_{L^q([0,T_0] ; L^r)}$.  By Duhamel's formula and Proposition \ref{inhomostr},
\begin{eqnarray*}
\|u-v\|_{L^q([0,T_0] ; L^r)} &= & \left\| \int^t_0 e^{i(t-\tau)\pa_x^2} (|u|^2u-|v|^2v) d\tau \right\|_{L^q([0,T_0] ; L^r)} \\
& \le & C\left\| |u|^2u-|v|^2v \right\|_{L^{\gamma}([0,T_0] ; L^{\rho} )}.
\end{eqnarray*}
By H\"older's inequality in the time variable, the norm in the right hand side is estimated by
\begin{eqnarray*}
&&\|u^2(\bar{u} -\bar{v})\|_{L^{\gamma}([0,T_0] ; L^{\rho} )} +\|u\bar{v}(u-v)\|_{L^{\gamma}([0,T] ; L^{\rho} )}+\|v\bar{v}(u-v)\|_{L^{\gamma}([0,T_0] ; L^{\rho} )}  \\
&\le& C T_0^{1-\frac{1}{p}} \left( \|u \|_{L^q([0,T];L^r)}^2 +\|u \|_{L^q([0,T];L^r)}\|v \|_{L^q([0,T];L^r)} 
+\|v \|_{L^q([0,T];L^r)}^2 \right) \|u-v\|_{L^q([0,T_0];L^r)}.
\end{eqnarray*}
We put $\eta_T \triangleq \max \left( \|u\|_{L^q([0,T];L^r)}, \|v\|_{L^q([0,T];L^r)} \right)$.  Then we have
\begin{equation}
\|u-v\|_{L^q([0,T] ; L^r)} \le 3CT_0^{1-\frac{1}{p}}  \eta_T^2 \|u-v\|_{L^q([0,T] ; L^r)} . \label{difference}
\end{equation}
Now we choose $T_0$ so that
\begin{equation*}
3CT_0^{1-\frac{1}{p}}  \eta_T^2 \le \frac{1}{2}.
\end{equation*}
Then $\|u-v\|_{L^q([0,T_0];L^r)}=0$ and thus $u(t)=v(t),\,\forall t \in [0,T_0]$.  In a similar manner, we can see that $u(t)=v(t),\,\forall t \in [T_0, 2T_0]$, since
the time interval $T_0$ can be determined depending only on $\eta_T$.  Repeating this argument, we finally get $u=v$ on $[0,T]$, which concludes the uniqueness
in  the space $\mathfrak{C}([0,T]\,; H^{\delta}_p)$.

The uniqueness assertion for the case (ii) of Theorem \ref{ULWP} can be treated in the same manner.  We let $3/2\le p\le 2$ and we assume $s>2/3-1/p$ and $s$ is sufficiently close to $2/3-1/p$.  For a sufficiently small 
$\delta>0$ we may write $s = 2/3-1/p+\delta$.  We set
\begin{equation*}
\frac{1}{q}=\frac{1}{6}+\frac{\delta}{2},\quad \frac{1}{r}=\frac{1}{3}-\delta,\quad \frac{1}{\rho}=1-\frac{3}{2}\delta (=\frac{3}{r}),\quad \frac{1}{\gamma}=\frac{5}{6}+\frac{3}{4}\delta.
\end{equation*}
Then it is easy to check that $q,r,\gamma,\rho$ satisfy the  assumption of Proposition \ref{inhomostr} if $\delta$ is small enough.  Moreover, 
$s,q,r$ satisfy the condition of Lemma \ref{Sobolev}, and thus
\begin{equation*}
\mathfrak{L}^{\infty} ([0,T] ;H^s_p ) \hookrightarrow L^q([0,T] ; L^r).
\end{equation*}
Then we get the desired uniqueness in $\mathfrak{C}([0,T] ; H^s_p)$ by considering two solutions $u,v\in \mathfrak{C}([0,T] ; H^s_p)$
with $u(0)=v(0)$ and estimating the difference $u-v$ in $L^q([0,T] ;L^r)$ as in case (i).

\section{ Proof of conditional well-posedness}

\subsection{Strichartz estimates}
In the last section we prove the conditional well-posedness.  We exploit the standard Strichartz technique.  
This approach is well known in the case where the initial data lie in the $L^2$ and $H^s$ spaces $(s\ge 0)$.  But it can also work
well for the case where data are in $L^p,\,4/3<p<2$. 
We begin with the estimates for the homogeneous equation.

\begin{prop} (\cite[Theorem 3.2]{Kato}) \label{PStr}
Let $1<p\le 2$ and let $q,r \in [2,\infty]$ be such that
\begin{equation*}
\frac{2}{q}+\frac{1}{r}=\frac{1}{p}
\end{equation*}
and
\begin{equation*}
\frac{1}{q}+\frac{1}{r}<\frac{1}{2}.
\end{equation*}
Then the following estimate holds true:
\begin{equation}
\|e^{it\pa_x^2} \phi \|_{L^q(\R; L^r(\R))} \le C\|\phi \|_{L^p(\R)}. \label{LPstr}
\end{equation}
In particular, for any $s \in \R$ the estimate
\begin{equation}
\|e^{it\pa_x^2} \phi \|_{L^q(\R; H^s_r(\R))} \le C\|\phi \|_{H^s_p(\R)} \label{HPstr}
\end{equation}
holds true.
\end{prop}

\medskip

The estimate (\ref{LPstr}), (\ref{HPstr}) are exploited to estimate the linear part of the corresponding integral equation.  Moreover, when $3/4<p\le 2$, stronger estimates are known:

\begin{prop} (\cite[Theorem 5]{107T})\label{FS}
Let $4/3<p\le 2$ and let $q,r \in [2,\infty]$ be such that
\begin{equation*}
\frac{2}{q}+\frac{1}{r}=\frac{1}{p}
\end{equation*}
and 
\begin{equation*}
0<\frac{1}{q}<\min \left( \frac{1}{2}-\frac{1}{r}, \frac{1}{4} \right).
\end{equation*}
Then 
\begin{equation}
\|e^{it\pa_x^2} \phi \|_{L^q (\R ; L^r(\R)} \le C\|\hat{\phi} \|_{L^{p'}(\R)} \label{FSI}
\end{equation}

\end{prop}

\bigskip

The diagonal case $q=r=3p$ of (\ref{FSI}) goes back to Fefferman and Stein \cite{Fefferman} and is used to obtain the 
well-posedness results for (\ref{NLS}) in \cite{Grunrock,CVV,107T}.  In this paper we essentially use (\ref{FSI}) to
deal with the nonlinear term.  Indeed, the following equivalent form of Proposition \ref{FS} is useful in the estimates of the Duhamel term
of the integral equation.

\begin{cor} (\cite[Proposition 2.6]{107jfa})\label{DFS}
Let $4/3<p \le 2$ and let $\gamma,\sigma$ be such that $(q,r):=(\gamma', \sigma')$ satisfy the assumption of Proposition \ref{FS}.  Then
\begin{equation*}
\left\| \int^t_0 e^{i(t-\tau)\pa_x^2} F(\tau) d\tau \right\|_{\mathfrak{L}^{\infty} ([0,T] ; L^p)}
\le C\| \tau^{\frac{1}{p}-\frac{1}{2}} F(\tau)\|_{L^{\sigma} ([0,T] ; L^{\rho} (\R))}.
\end{equation*}

In particular, for any $s \in \R$ the estimate
\begin{equation}
\left\| \int^t_0 e^{i(t-\tau)\pa_x^2} F(\tau) d\tau \right\|_{\mathfrak{L}^{\infty} ([0,T] ; H^s_p)}
\le C\| \tau^{\frac{1}{p}-\frac{1}{2}} F(\tau) \|_{L^{\sigma} ([0,T] ; H^s_{\rho} (\R))} \label{DFSest}
\end{equation}
holds true.
\end{cor}

\subsection{ Proof of Proposition \ref{LWP}}\,\, We first introduce several particular exponents.  We set
\begin{equation*}
\frac{1}{q}=\frac{1}{2p}-\frac{1}{8}-\frac{s}{4},\,\,\,\frac{1}{r}=\frac{1+2s}{4},\,\,\,\frac{1}{\rho}=\frac{3}{4}-\frac{s}{2}\left(=\frac{3}{r}-2s\right),\,\,\, \frac{1}{\gamma}=\frac{5}{8}+\frac{s}{4}+\frac{1}{2p},\,\,\, \frac{1}{\sigma}=\frac{9}{8}-\frac{1}{2p}+\frac{s}{4}\left(=\frac{1}{q'}\right).
\end{equation*}

Then it is easy to check that $(q,r,\gamma,\sigma)$, $(q,r)$, and $(\rho,\sigma)$ satisfy the assmption of Proposition \ref{inhomostr}, Proposition \ref{PStr}, and Corollary \ref{DFS} if
$4/3<p\le 2$ and $0<s<3/2-2/p$.  Let $T>0$.  We claim that
\begin{equation}
\left\|\int^t_0 e^{i(t-\tau)\pa_x^2} (u_1 u_2 \overline{u_3}) d\tau  \right\|_{L^q ([0,T] ; H^s_r(\R) )}\le 
CT^{1+s-\frac{1}{p}} \prod_{j=1}^3 \| u_j \|_{L^q([0,T];H^s_r(\R))  }, \label{Snonlinearest1}
\end{equation}
and
\begin{equation}
\left\|\int^t_0 e^{i(t-\tau)\pa_x^2} (u_1 u_2 \overline{u_3}) d\tau  \right\|_{\mathfrak{L}^{\infty} ([0,T] ; H^s_p(\R))}\le 
CT^{1+s-\frac{1}{p}} \prod_{j=1}^3 \| u_j \|_{L^q([0,T];H^s_r(\R))  }, \label{Snonlinearest2}
\end{equation}
for any $u_j \in L^q([0,T];H^s_r(\R)),\,\,j=1,2,3$, where $C$ is independent of $T$.  We first prove (\ref{Snonlinearest1}).  Applying (\ref{inhomostrest}) with  $F=\langle D \rangle^s(u_1u_2 u_3)$, we have
\begin{eqnarray*}
\left\|\int^t_0 e^{i(t-\tau)\pa_x^2} (u_1 u_2 \overline{u_3}) d\tau  \right\|_{L^{q} ([0,T] ; H^s_r)}\le 
C \| u_1 u_2 \overline{u_3} \|_{ L^{\gamma} ([0,T] ; H^s_{\rho} (\R)) }.
\end{eqnarray*}
We estimate the right hand side.  By the fractional Leibniz rule (see e.g. \cite{GO})
\begin{equation*}
\|u_1 u_2 \overline{u_3} \|_{H^s_{\rho}(\R)}
\le C\left( \|u_1\|_{H_r^s(\R)} \|u_2 \overline{u_3} \|_{L^{\kappa}(\R)} +\|u_1\|_{L^{2\kappa}(\R)} \|u_2 \overline{u_3}
\|_{H^s(\R)} \right),
\end{equation*}
where
\begin{equation*}
\frac{1}{\kappa}=\frac{1}{\rho}-\frac{1}{r}=\frac{1}{2}-s.
\end{equation*}
For the first term in the right hand side we have by Schwartz's inequality and Sobolev's embedding, we have
\begin{equation*}
\|u_1\|_{H_r^s(\R)} \|u_2 \overline{u_3} \|_{L^{\kappa}(\R)} 
\le \|u_1\|_{H_r^s(\R)} \|u_2 \|_{L^{2\kappa}(\R)} \|u_3 \|_{L^{2\kappa} (\R)}
\le C \prod_{j=1}^3 \|u_j\|_{H_r^s(\R)}.
\end{equation*}
Noting the relation
\begin{equation*}
\frac{1}{2}=\left(\frac{1}{4}-\frac{s}{2} \right)+\left(\frac{1}{4}+\frac{s}{2} \right)=\frac{1}{2\kappa}+\frac{1}{r},
\end{equation*}
the second term can be estimated as
\begin{equation*}
\|u_1\|_{L^{2\kappa}(\R)} \|u_2 \overline{u_3}
\|_{H^s(\R)} \le C\|u_1 \|_{L^{2\kappa}(\R)} \left( \|u_2 \|_{H_r^s(\R)} \|u_3\|_{L^{2\kappa}(\R)} 
+ \|u_2 \|_{L^{2\kappa}(\R)} \|\overline{u_3} \|_{H^s_r (\R)}\right),
\end{equation*}
the right hand side of which can be controlled by $\|u_1\|_{H_r^s} \|u_2\|_{H_r^s}\|u_3\|_{H_r^s}$ by Sobolev's embedding.  Consequently, we have
\begin{equation*}
\| u_1 u_2 \overline{u_3} \|_{ L^{\gamma} ([0,T] ; H^s_{\rho} (\R)) }
\le C\left\| \prod_{j=1}^3 \|u_j\|_{H_r^s(\R)} \right\|_{L^{\gamma}([0,T])}
\le CT^{1+s-\frac{1}{p}} \prod_{j=1}^3 \| u_j \|_{L^q([0,T];H^s_r(\R))  }
\end{equation*}
by H\"older's inequality in the time variable.  This concludes the proof of (\ref{Snonlinearest1}).  To prove (\ref{Snonlinearest2})
we use (\ref{DFSest}) to obtain
\begin{eqnarray*}
\left\|\int^t_0 e^{i(t-\tau)\pa_x^2} (u_1 u_2 \overline{u_3}) d\tau  \right\|_{\mathfrak{L}^{\infty} ([0,T] ; H^s_p)}&\le & C \| \tau^{\frac{1}{p}-\frac{1}{2}} u_1(\tau) u_2(\tau) \overline{u_3(\tau)} \|_{ L^{\sigma} ([0,T] ; H^s_{\rho} (\R)) }.
\end{eqnarray*}
As seen in the proof of (\ref{Snonlinearest1}), the norm $\| \tau^{\frac{1}{p}-\frac{1}{2}} u_1(\tau) u_2(\tau) \overline{u_3(\tau)} \|_{ H^s_{\rho} }$ can be controlled by $C\tau^{\frac{1}{p}-\frac{1}{2}}\|u_1(\tau) \|_{H_r^s} \|u_2 (\tau )\|_{H_r^s}\|u_3 (\tau )\|_{H_r^s}$.  Thus the right hand side of the above 
inequality is bounded from above by
\begin{eqnarray*}
&&C \left\| \tau^{\frac{1}{p}-\frac{1}{2}} \prod_{j=1}^3 \| u_j \|_{L^q([0,T];H^s_r(\R))} \right\|_{ L^{\sigma}([0,T]) } \\
&\le & C\| \tau^{\frac{1}{p}-\frac{1}{2}} \|_{L^{\frac{\rho q}{q-3\rho}} ([0,T])}
\prod_{j=1}^3 \| u_j \|_{L^q([0,T];H^s_r(\R))  } \\
&\le & CT^{1+s-\frac{1}{p}} \prod_{j=1}^3 \| u_j \|_{L^q([0,T];H^s_r(\R))  },
\end{eqnarray*}
where we have used H\"older's inequality in the time variable.  This proves (\ref{Snonlinearest2}).

Now we establish a local solution of the corresponding integral equation
\begin{equation*}
u(t)=e^{it\pa_x^2}\phi +i\int^t_0 e^{i(t-\tau)\pa_x^2} (|u|^2u) d\tau.
\end{equation*}
We find a fixed point of the operator
\begin{equation*}
Su \triangleq e^{it\pa_x^2}\phi +i\int^t_0 e^{i(t-\tau)\pa_x^2} (|u|^2u) d\tau
\end{equation*}
in a suitable closed subset of $\mathfrak{L}^{\infty} ([0,T] ; H^s_p) \cap L^q([0,T];H^s_r)$.  For $T,R>0$ we define
\begin{equation*}
\mathcal{V}_{T,R} \triangleq \{ u \in \mathfrak{L}^{\infty} ([0,T] ; H^s_p) \cap L^q([0,T];H^s_r)\,\,
|\, \|u\|_{ \mathfrak{L}^{\infty} ([0,T] ; H^s_p) } +\|u\|_{L^q([0,T];H^s_r)} \le R\,\}
\end{equation*}
equipped with the distance
\begin{equation*}
d(u,v) \triangleq \|u-v\|_{ \mathfrak{L}^{\infty} ([0,T] ; H^s_p) } +\|u-v\|_{L^q([0,T];H^s_r)} .
\end{equation*}
By (\ref{HPstr}), (\ref{Snonlinearest1}), (\ref{Snonlinearest2}) we have for any $u \in \mathcal{V}_{T,R}$
\begin{eqnarray*}
\| Su \|_{\mathfrak{L}^{\infty} ([0,T] ; H^s_p)  }  +\|Su \|_{ L^q([0,T];H^s_r }& \le & C\left( \|\phi \|_{H_p^s} +T^{1+s-\frac{1}{p}} \|u \|_{ L^q([0,T];H^s_r )}^3  \right)\\
&\le &  C\left( \|\phi \|_{H_p^s} +T^{1+s-\frac{1}{p}} R^3  \right).
\end{eqnarray*}
Now we take $T,R$ so that
\begin{equation*}
R=2C\|\phi \|_{H^s_p},\qquad T^{1+s-\frac{1}{p}} \le \frac{R^{-2}}{4C}.
\end{equation*}
Then the right hand side is smaller than $R/2+R/4 \le R$ and thus $S:\mathcal{V}_{T,R} \to \mathcal{V}_{T,R}$ is well defined.  Similarly, for $u,v\in \mathcal{V}_{T,R}$ 
\begin{eqnarray*}
d(Su,Sv) &\le & CT^{1+s-\frac{1}{p}} ( \| u\|^2_{ L^q([0,T];H^s_r )} +\| u\|_{ L^q([0,T];H^s_r) }\| v\|_{ L^q([0,T];H^s_r) }+ \| v\|^2_{ L^q([0,T];H^s_r) } )\\
 && \times \| u-v\|_{ L^q([0,T];H^s_r) }    \\
&\le & C \times \frac{R^{-2}}{4C} \times 3R^2 \times \| u-v\|_{ L^q([0,T];H^s_r) }  \\
&\le & \frac{3}{4}d(u,v),
\end{eqnarray*}
from which we see that  $S:\mathcal{V}_{T,R} \to \mathcal{V}_{T,R}$ is a contraction mapping.  
Consequently, a local solution $v\in \mathfrak{L}^{\infty} ([0,T] ; H^s_p)\cap  L^q([0,T];H^s_r )$ with 
$T\sim \|\phi\|_{H^s_p}^{-\frac{2}{1+s-\frac{1}{p}}}$ can be constructed
by the standard fixed point theorem.  Finally a difference estimate similar to the proof of Theorem \ref{ULWP}
give the uniqueness of the solution in the space $L^q([0,T] ;H_r^s)$ and the continuous dependence on data.

\end{document}